\documentclass[10pt]{amsart}
\usepackage[top=3cm,bottom=3cm,left=4cm,right=4cm]{geometry}               
\usepackage{amsmath,amscd,amssymb,amsthm}
\usepackage[english]{babel}
\usepackage{mathtools}
\usepackage{enumerate}
\usepackage{setspace}
\usepackage{mathrsfs}
\usepackage[numbers]{natbib}
\usepackage[hidelinks]{hyperref}

\DeclarePairedDelimiter{\abs}{\lvert}{\rvert}

\DeclareMathOperator{\Div}{Div}
\DeclareMathOperator{\Char}{char}
\DeclareMathOperator{\supp}{supp}

\def\vF{\mathbb{F}}

\def\vZ{\mathbb{Z}}

\def\vR{\mathbb{R}}

\def\vD{\mathbb{D}}

\def\vDs{\overline{\mathbb{D}}}
\def\vDi{\underline{\mathbb{D}}}

\def\cO{\mathcal{O}}
\def\cS{\mathcal{S}}
\def\cQ{\mathcal{Q}}
\def\cD{\mathcal{D}}
\def\cL{\mathcal{L}}
\def\cC{\mathcal{C}}

\newtheorem{theorem}{Theorem}

\newtheorem{corollary}[theorem]{Corollary}

\theoremstyle{definition}

\newtheorem{remark}[theorem]{Remark}

\makeindex

 \author[G. Micheli]{Giacomo Micheli}
 \address{Institute of Mathematics\\ 
 University of Zurich\\
 Winterthurerstrasse 190\\
 8057 Zurich, Switzerland\\
 }
 \email{giacomo.micheli@math.uzh.ch}

 \author[R. Schnyder]{Reto Schnyder}
 \address{Institute of Mathematics\\ 
 University of Zurich\\
 Winterthurerstrasse 190\\
 8057 Zurich, Switzerland\\
 }
 \email{reto.schnyder@math.uzh.ch}

\thanks{The authors were supported in part by Swiss National Science Foundation grant number 149716 and \emph{Armasuisse}}

\date{}
\title{On the density of coprime $m$-tuples\linebreak\ over holomorphy rings}
\begin{document}

\begin{abstract}
Let $\vF_q$ be a finite field, $F/\vF_q$ be a function field of genus $g$ having full constant field $\vF_q$, $\cS$ a set of places of $F$ and $H$ the holomorphy ring of $\cS$.
In this paper we compute the density of coprime $m$-tuples of elements of $H$.
As a side result, we obtain that whenever the complement of $\cS$ is finite, the computation of the density can be reduced to the computation of the $L$-polynomial of the function field.
In the rational function field case, classical results for the density of coprime $m$-tuples of polynomials are obtained as corollaries.
\end{abstract}

\maketitle

\smallskip
\noindent \textbf{Keywords}: Function fields, Density, Polynomials, Riemann-Roch spaces, Zeta function. \\
\smallskip
\noindent \textbf{MSC}: 11R58, 11M38, 11T06 

\section{Introduction}
\label{sec:introduction}

Let $\vF_q$ be a finite field with $q$ elements and let $F$ be an algebraic function field\footnote{In this note we will mostly use the language and notation of~\cite{bib:stichtenoth2009algebraic}.
}  having full constant field $\vF_q$. Let $\cC$ be the set of places of $F$  and $\cS \subsetneq \cC$ be a non-empty proper subset.
The holomorphy ring of $\cS$ is $H = \bigcap_{P \in \cS} \cO_P$, where $\cO_P$ is the valuation ring of the place $P$.

In what follows we will say that an $m$-tuple of elements of $H$ is coprime if its components generate the unit ideal in $H$ (in analogy to the case of the ring of integers in~\cite{bib:MicheliCesaro}).
In this paper we define a notion of density for subsets of $H^m$, using Moore-Smith convergence for nets~\cite[Chapter~2]{bib:kelley1955general}.
We then wish to study the density of the set of coprime $m$-tuples in $H$, considered as a subset of $H^m$.

The special case $F=\vF_q(x)$ and $H=\bigcap_{P\neq P_\infty} \cO_P=\vF_q[x]$ has been studied for $m=2$ in~\cite{bib:sugita2007probability} and more generally in~\cite{bib:guo2013probability}.
We will explain how to interpret the densities presented in these papers as particular cases of our general framework.
In fact, using the Riemann-Roch Theorem and the absolute convergence of the Zeta function of $F$, we are able to show that the density of coprime $m$-tuples of elements of a holomorphy ring exists and is equal to $1\over Z_H(q^{-m})$, where $Z_H$ is the Zeta function of the holomorphy ring, which will be defined in Section~\ref{sec:density}.

Finally, we provide an example  in the case of the affine ring of coordinates of an elliptic curve to show a concrete application of the main result.

The results in this paper provide a function field version of a classical result for the ring of integers, where the natural density of the set of coprime $m$-tuples of $\vZ$ is proven to be equal to $\frac{1}{\zeta(m)}$, $\zeta$ being the classical Riemann zeta function (see for example~\cite{bib:nymann1972probability}).
Similar results also hold in the rings of integers of arbitrary number fields (see~\cite{bib:MicheliCesaro} and~\cite{bib:BS}).

\subsection{Notation}
\label{sec:notation}

Let $F/\vF_q$ be an algebraic function field with full constant field $\vF_q$, 
let $g$ be the genus of $F$, and let $\cC$ be the set of its places.
Let $H$ be the holomorphy ring of a nonempty set of places $\cS\subsetneq \cC$.
For a fixed positive integer $m$, we wish to study the set of
coprime $m$-tuples of elements of the ring $H$. Let us denote this set by $U$:
\begin{equation*}
	U \coloneqq \{ f=(f_1,\dots,f_m) \in H^{m} \mid I_f = H \},
\end{equation*}
where $I_f$ denotes the ideal of $H$ generated by the set $\{f_1,\dots f_m\}$.

Define furthermore $\cD \coloneqq \{D \in \Div(F) \mid D \ge 0 \wedge \supp(D)
\subseteq \cC \setminus \cS\}$, the set of positive divisors supported away from $\cS$.
For any divisor $D$, the Riemann-Roch space associated to $D$ is defined as in \cite[Def. 1.4.4]{bib:stichtenoth2009algebraic} by
\[\cL(D):=\{f\in F \mid v_P(f)+v_P(D)\geq 0\;\forall P\in \cC\}\cup \{0\},\]
where $v_P$ is the valuation associated to the place $P$.
It follows that
\begin{equation*}
	H = \bigcup_{D \in \cD} \cL(D),
\end{equation*}
Recall that we have a bijection between  the set of places $\cS$ and the maximal ideals of $H$ given by $P\mapsto P\cap H\eqqcolon P_H$ (see for example~\cite[Proposition~3.2.9]{bib:stichtenoth2009algebraic}).
In analogy to the natural density of integers, we define the \emph{superior density} of a subset $L \subseteq H^{m}$
as
\begin{equation}\label{eq:supdens}
	\vDs(L) \coloneqq \limsup_{D \in \cD} 
	\frac{\abs{L \cap \cL(D)^{m}}}{\abs{\cL(D)^{m}}}.
\end{equation}
This limit can be defined via Moore-Smith convergence~\cite[Chapter~2]{bib:kelley1955general}. To be precise, the set of divisors $\cD$ with the usual partial order $\leq$ is a directed set, so the map from $\cD$ to the topological space $\vR$ defined as 
\[D\mapsto \frac{\abs{L \cap \cL(D)^{m}}}{\abs{\cL(D)^{m}}}\]
is a net.
Now, since $\vR$ is Hausdorff, the definition in \eqref{eq:supdens} is well posed.
Analogously one can define the \emph{inferior density} as
\begin{align*}
	\vDi(L) &\coloneqq \liminf_{D \in \cD} 
	\frac{\abs{L \cap \cL(D)^{m}}}{\abs{\cL(D)^{m}}}. 
\end{align*}
Moreover, whenever $\vDs(L) = \vDi(L)$, we call this value the \emph{density} of $L$ and denote it by $\vD(L)$. In the case in which $\cC\setminus \cS$ is finite, an analogous definition of density can also be found in \cite[Section 8]{bib:poonen}.

\section{The density of $U$}
\label{sec:density}

Recall that the Zeta function of the function field $F$ is given by
\[Z_F(T)\coloneqq\prod_{P\in \cC} \left(1-{T^{\deg(P)}}\right)^{-1}\]
for $0 < T < q^{-1}$. Analogously, we define the Zeta function of the holomorphy ring $H$ corresponding to the set of places $\cS$ as
\[Z_H(T)\coloneqq\prod_{P\in \cS} \left(1-{T^{\deg(P)}}\right)^{-1}.\]

We will now state our main result.

\begin{theorem}\label{thm:main}
The density of the set of coprime tuples of length $m\geq 2$ of the holomorphy ring $H$ is	 $\frac{1}{Z_H(q^{-m})}$.
\end{theorem}
\begin{proof}
We first enumerate the set of places of $\cS=\{Q_1,Q_2,\dots,Q_t,\dots\}$. Let us define
\[U_t\coloneqq\{f=(f_1,\dots f_m)\in H^m \mid I_f\nsubseteq (Q_i)_H \quad \forall i\in \{1,\dots,t\} \}\]
and notice that $U_t\supseteq U$.
Observe that the condition $I_f\nsubseteq (Q_i)_H$ is equivalent to the fact that for each $i$ there exists at least one $f_j$ that does not belong to $Q_i$.
Consider now the projection $\pi\colon H \to H/((Q_1)_H \cdots (Q_t)_H)$
and observe that
\[	
H/((Q_1)_H \cdots (Q_t)_H) \cong \prod_{i=1}^{t} H/(Q_i)_H \cong \prod_{i=1}^{t} \vF_{q^{\deg Q_i}},
\]
by the Chinese remainder theorem over the ideals $(Q_i)_H$.
This gives us a homomorphism
\[\phi\colon H \longrightarrow \prod_{i=1}^{t}\vF_{q^{\deg Q_i}},\]
which we can extend to $m$-tuples by
\[\widehat\phi\colon H^m \longrightarrow \prod_{i=1}^{t}\vF^m_{q^{\deg Q_i}}.\]
By construction, this homomorphism satisfies
\begin{equation}\label{eq:Et}
U_t=\widehat\phi^{-1}\left(\prod_{i=1}^{t}(\vF^m_{q^{\deg Q_i}}\setminus \{0\})\right).
\end{equation}

Consider now a divisor $D \in \cD$.
We wish to count the number of elements in $U_t \cap \cL(D)^{m}$.
First, we will show that $\phi$ maps $\cL(D)$ surjectively onto
$\prod_{i=1}^{t} \vF_{q^{\deg Q_i}}$ if $\deg D$ is large enough.

For this, note that the image of $\cL(D)$ under $\pi$ is
$\cL(D)/(\cL(D) \cap ((Q_1)_H \cdots (Q_t)_H))$.
The space
\begin{equation*}
	\cL(D) \cap ((Q_1)_H \cdots (Q_t)_H) =
	\cL(D) \cap Q_1 \cap \cdots \cap Q_t
\end{equation*}
consists of all elements in
$\cL(D)$ with at least a root at each $Q_i$, so it is equal to
$\cL(D - \sum_{i=1}^{t} Q_i)$
(note that the $Q_i$ cannot be in the support of $D$).
Hence, its dimension as an $\vF_q$-vector space is
$\ell(D - \sum_{i=1}^{t} Q_i)$, which is equal to
$\deg D - \sum_{i=1}^{t} \deg Q_i + 1 - g$ if $\deg D$ is large enough by Riemann-Roch Theorem.
On the other hand, the dimension of $\cL(D)$ is then
$\ell(D) = \deg D + 1 - g$, and so the image has dimension
$\sum_{i=1}^{t} \deg Q_i$, the same as $\prod_{i=1}^{t} H/(Q_i)_H$.
Therefore $\pi$ and $\phi$ restricted to $\cL(D)$ are surjective.

We can now count the elements of $U_t \cap \cL(D)^m$ using \eqref{eq:Et}.
As we have just seen, the dimension of the kernel of $\phi$ restricted
to $\cL(D)$ is $\ell(D - \sum_{i=1}^{t} Q_i)$, so each element of $\prod_{i=1}^{t}(\vF_{q^{\deg Q_i}}^m \setminus \{0\})$ is the image under $\widehat\phi$ of exactly $q^{m \ell(D - \sum_{i=1}^{t}Q_i)}$ elements of $U_t \cap \cL(D)^m$.
Hence we get
\[
	\frac{\abs{U_t \cap \cL(D)^{m}}} {\abs{\cL(D)^{m}}}
	= q^{m(\ell(D - \sum_{i=1}^{t} Q_i) - \ell(D))} \cdot
	\prod_{i=1}^{t} (q^{m \deg Q_i} - 1)=\prod_{i=1}^{t} (1 - q^{-m \deg Q_i})
\]
if $\deg D$ is large enough.
It follows that the density of $U_t$ is well-defined and equals  \[\vD(U_t)=\lim_{D\in \cD}{\frac{\abs{U_t \cap \cL(D)^{m}}} {\abs{\cL(D)^{m}}}}= \prod_{i=1}^{t} (1 - q^{-m \deg Q_i}).\]
Since $U \subseteq U_t$, it follows that $\vDs(U) \le \vD(U_t)$.

To get an estimate in the other direction, let us write 
$(\cL(D)\cap U)\cup(\cL(D)\cap(U_t\setminus U))=\cL(D)\cap U_t$.
We have
\[\vDi(U)=\liminf_{D\in \cD}\frac{\abs{U\cap\cL(D)^m}}{q^{m\ell(D)}}\geq\lim_{D\in \cD}\frac{\abs{U_t\cap\cL(D)^m}}{q^{m\ell(D)}} -\limsup_{D\in \cD} \frac{\abs{(U_t\setminus U)\cap\cL(D)^m}}{q^{m\ell(D)}},\]
hence we have the inequalities
\begin{equation}\label{eq:fundeq}
\vD(U_t) \ge \vDs(U) \ge \vDi(U) \ge \vD(U_t)-\limsup_{D \in \cD} \frac{\abs{(U_t\setminus U)\cap\cL(D)^m}}{q^{m\ell(D)}}.
\end{equation}
Now, passing to the limit in $t$, we get that $\displaystyle{\lim_{t\rightarrow \infty}\vD(U_t)}=1/Z_H(q^{-m})$.
Therefore it remains to prove that
\[\lim_{t\rightarrow \infty}\limsup_{D \in \cD} \frac{\abs{(U_t\setminus U)\cap\cL(D)^m}}{q^{m\ell(D)}}=0.\]
In order to prove the last claim, let us denote by $\cQ_t$ the set $\{Q_1,\dots,Q_t\}$. Notice that if $A\in U_t\setminus U$, then there exists $P\in \cS\setminus \cQ_t$ for which $I_A\subseteq P_H$.
We get the following inclusion:
\begin{equation*}
		U_t \setminus U
		\subseteq
		\bigcup_{\mathclap{P \in \cS \setminus \cQ_t}} \{A \in H^{m} \mid
			I_A \subseteq P_H\}
		=
		\bigcup_{\mathclap{P \in \cS \setminus \cQ_t}} P_H^{m},
	\end{equation*}
where by $P_H^{m}$ we mean the Cartesian product of $m$ copies of the ideal $P_H$.
Fix now a divisor $D \in \cD$.
	It follows that
\begin{equation*}
		(U_t \setminus U) \cap \cL(D)^{m}
		\subseteq
		\bigcup_{\mathclap{P \in \cS \setminus \cQ_t}} (P_H \cap \cL(D))^{m}
		=
		\bigcup_{\mathclap{P \in \cS \setminus \cQ_t}} \cL(D - P)^{m}
		=
		\bigcup_{\mathclap{\substack{P \in \cS \setminus \cQ_t \\ \deg P \le \deg D}}}\cL(D - P)^{m}.
\end{equation*}
The last equality holds because $\cL(D - P) = 0$ if $\deg D - \deg P < 0$.
	With this containment, we can now estimate the last term of \eqref{eq:fundeq}:
	\begin{align*}
		\limsup_{D \in \cD} \frac{\abs{\cL(D)\cap(U_t\setminus U)}}{q^{m\ell(D)}}
		&\le
		\limsup_{D \in \cD} \abs[\Big]{
			\bigcup_{\mathclap{\substack{P \in \cS \setminus \cQ_t \\ \deg P \le \deg D}}}
			\cL(D - P)^{m}
		} \cdot q^{-m \ell(D)}
		\\&\le
		\limsup_{D \in \cD}
		\sum_{\mathclap{\substack{P \in \cS \setminus \cQ_t \\ \deg P \le \deg D}}}
		\abs[\Big]{
			\cL(D - P)^{m}
		} \cdot q^{-m \ell(D)}
		\\&=
		\limsup_{D \in \cD}
		\sum_{\mathclap{\substack{P \in \cS \setminus \cQ_t \\ \deg P \le \deg D}}}
		q^{m(\ell(D-P) - \ell(D))}		
\end{align*}
By the Riemann-Roch Theorem, we have that $\ell(D)\geq \deg(D)+1-g$ and $\ell(D-P)\leq \deg(D-P)+1$, since $\deg(D-P)\geq 0$~\cite[Eq.~1.21 and Theorem~1.4.17]{bib:stichtenoth2009algebraic}. It follows that the above is less or equal to
\[		
		\limsup_{D \in \cD}
		\sum_{\mathclap{\substack{P \in \cS \setminus \cQ_t \\ \deg P \le \deg D}}}
		q^{m(g - \deg P)}
		=
		\limsup_{d \to \infty} q^{gm}
		\sum_{\mathclap{\substack{P \in \cS \setminus \cQ_t \\ \deg P \le d}}}
		(q^{-m})^{\deg P}
		=
		q^{gm}
		\sum_{\mathclap{P \in \cS \setminus \cQ_t}} (q^{-m})^{\deg P}.
\]
Observe that $\sum_{P \in \cS \setminus \cQ_t}q^{-m\deg P}$
is the tail of a subseries  of the Zeta function of $F$ evaluated at $q^{-m}<q^{-1}$, which is absolutely convergent (see for example~\cite[Chapter 3]{bib:Mor}). As $t$ goes to infinity, it converges to $0$, from which our claim follows.
\end{proof}

\section{Consequences}
The reader should observe that in Theorem~\ref{thm:main} both $\cS$ and $\cC\setminus \cS$ could possibly be infinite and the result will still hold. Nevertheless, the density depends on the Zeta function of the holomorphy ring, which may be hard to compute.
First of all notice that this is not the case when $\cS$ is finite since under this condition $Z_H$ is a finite product.
The following immediate corollary covers the case in which $\cC\setminus \cS$ is finite.
\begin{corollary}\label{thm:comp}
Let $F$ be a function field, $\cS$ a set of places of $F$ and $H$ the holomorphy ring of $\cS$. Let $L_F(T)$ be the $L$-polynomial of $F$.
Then 
\[Z_H(q^{-m})=\frac{L_F(q^{-m})}{(1-q^{-m})(1-q^{-m+1})} \prod_{P\in \cC\setminus \cS} \left(1-\frac{1}{q^{\deg(P)m}}\right).\]
\end{corollary}
\begin{proof}
The corollary follows from Theorem~\ref{thm:main}, the definition of $Z_H$ and the expression of the Zeta function of $F$ in terms of the $L$-polynomial.
\end{proof}
\begin{remark}\label{rem:findata}
Observe now that in the case where $\cC\setminus \cS$ is finite, the density of coprime $m$-tuples of $H$  depends only on the following finite data: the degrees of the places in $\cC \setminus \cS$ and the $L$-polynomial of the function field, which again only depends on the $\vF_{q^i}$-rational points of the curve associated to the function field for $i\in\{1,\dots,g\}$ (see for example~\cite[Corollary~5.1.17]{bib:stichtenoth2009algebraic}).
\end{remark}

\subsection{An example}
Let $\Char(\vF_q)\neq 2,3$ for simplicity.
Let $a,b\in \vF_q$ and $p(x,y)=y^2-x^3-ax-b$ be a polynomial defining an elliptic curve $E$ over  $\vF_q$.
Let us define
\[A(E)\coloneqq\vF_q[x,y]/(p(x,y)).\]
Let $E(\vF_q)$ denote the set of (projective) $\vF_q$-rational points of $E$ (i.e.\@ the places of degree one of the function field of $E$).
\begin{corollary}\label{thm:E}
The density of $m$-tuples of coprime elements of $A(E)$ is 
\begin{equation}\label{eq:equationind}
\vD(U)=\frac{1-q^{-m+1}}{1+a_q q^{-m}+q^{-2m+1}}
\end{equation}
where $a_q=q+1-\abs{E(\vF_q)}$.
\end{corollary}
\begin{proof}
Observe that the Zeta function of an elliptic curve is
\[Z_E(T)=\frac{1+a_q T+qT}{(1-q T)(1-T)}.\]
The result follows from Theorem~\ref{thm:main} applied to the holomorphy ring $A(E)=\bigcap_{P\neq P_\infty} \cO_P$ where $P_\infty$ is the place at infinity of $E$ with respect to $p(x,y)$.
\end{proof}
\begin{remark}
The reader should notice that \eqref{eq:equationind} depends only on the number of $\vF_q$-rational points of $E$, since the genus of $E$ equals one (see Remark~\ref{rem:findata}).
The probabilistic interpretation of Corollary~\ref{thm:E} is the following: select uniformly at random $m$ elements of $A(E)$ of degree at most $N$, then the probability that they generate the unit ideal in $A(E)$ approaches $\displaystyle{\frac{1-q^{-m+1}}{1+a_q q^{-m}+q^{-2m+1}}}$ as $N\rightarrow \infty$.
\end{remark}

\subsection{The case $F=\vF_q(x)$}
In the remaining part of this section we show how the results~\cite[Theorem~1]{bib:sugita2007probability} and~\cite[Remark~4]{bib:guo2013probability} about coprime $m$-tuples over $\vF_q[x]$ fit in our framework.

Denote by $P_\infty$ the place at infinity of the function field $\vF_q(x)$.
It is easy to see that the definition
of density for $\vF_q[x]$ given in~\cite{bib:guo2013probability,bib:sugita2007probability} agrees with ours for $H = \vF_q[x]=
\bigcap_{P\neq P_\infty} \cO_P$.
Hence, we get~\cite[Remark~4]{bib:guo2013probability} as a corollary to Theorem~\ref{thm:main}, while~\cite[Theorem~1]{bib:sugita2007probability} is simply the special case $m = 2$:
\begin{corollary}\label{thm:corrows}
Let $m>1$ be an integer.
The density of coprime $m$-tuples over $\vF_q[x]$ is 
\[\vD(U)=1-\frac{1}{q^{m-1}}.\]
\end{corollary}
\begin{proof}
It is enough to notice that the Zeta function of the function field $\vF_q(x)$ (i.e.\@ the Zeta function of the projective line) is 
\[Z_{\vF_q(x)}(T)=\frac{1}{(1-T)(1-qT)}\]
and then the Zeta function of the holomorphy ring
$\vF_q[x]=\bigcap_{P\neq P_\infty} \cO_P$
is \[Z_{\vF_q[x]}(T)=\frac{1}{1-qT}.\]
The claim follows by inverting the expression above and evaluating at $q^{-m}$. 
\end{proof}

\section*{Acknowledgements}
The authors want to thank Andrea Ferraguti for useful discussions and suggestions.

\bibliographystyle{plainnat}
\bibliography{biblio}{}

\end{document}